\documentclass{article}
      
\usepackage{latexsym}  
\usepackage{amsmath}
\usepackage{amsthm} 
\usepackage{amssymb}  
\usepackage{verbatim} 
\usepackage{dsfont}
\usepackage{graphicx}
\usepackage{epsfig}

\newtheorem{Counter}{!!!!!!!}

\newtheorem{Thm}[Counter]{Theorem}

\newtheorem{Prop}[Counter]{Proposition}
\newtheorem{Cor}[Counter]{Corollary}
\newtheorem{Lem}[Counter]{Lemma}
\newtheorem{Rem}[Counter]{Remark}

\begin{document}

\title {Free knot linear interpolation and the Milstein scheme for stochastic differential equations}

\author{
Mehdi Slassi \thanks{Fachbereich Mathematik, Technische Universit\"at Darmstadt, Schlo\ss gartenstra\ss e 7, 64289 Darmstadt, Germany}
}

%\author{
%Mehdi Slassi\thanks{Fachbereich Mathematik, Technische Universit\"at
%Darmstadt, Schlo\ss gartenstra\ss e 7, 64289 Darmstadt, Germany}
%}

\date{\today}

\maketitle 

\begin{abstract}
The main purpose of this paper is to give a solution to a long-standing unsolved problem concerning the pathwise strong approximation of stochastic differential equations with respect to the global error in the $L_{\infty}$-norm. Typically, one has average supnorm error of order $\left(\ln k/k\right)^{1/2}$ for standard approximations of SDEs with $k$ discretization points, like piecewise interpolated Ito-Taylor schemes. On the other hand there is a lower bound, which indicates that the order  $1/\sqrt{k}$ is best possible for spline approximation of SDEs with $k$ free knots. The present paper deals with the question of how to get an implementable method, which achieves the order $1/\sqrt{k}$. Up to now, papers with regard to this issue give only pure existence results and so are inappropriate for practical use. In this paper we introduce a nonlinear method for approximating a scalar SDE. The method combines a Milstein scheme with a pieceweise linear interpolation of the Brownian motion with free knots and is easy to implement. Moreover, we establish sharp lower and upper error bounds with specified constants which exhibit the influence of the coefficients of the equation. 
\medskip

\end{abstract} 
\section{Introduction}
\let\languagename\relax 
Consider a scalar stochastic differential equation (SDE)
\begin{equation}
dX\left(t\right)=a\left(t,X\left(t\right)\right)dt+\sigma\left(t,X\left(t\right)\right)dW\left(t\right),\qquad t\geq 0 \label{eq1}\end{equation}
with initial value $X(0)$. Here $W=\left(W\left(t\right)\right)_{t\geq 0}$
denotes a one-dimensional Brownian motion on a probability space $\left(\Omega,\mathcal{F},\mathbb{P}\right)$. We study pathwise approximation of equation (\ref{eq1}) on the unit interval by polynomial splines with free knots.

For $k\in\mathbb{N}$ and $r\in\mathbb{N}$ we let $\Pi_{r}$
denote the set of polynomials of degree at most $r$ and we consider
the space $\Phi_{k,r}$ of polynomial splines $\varphi$ of degree
at most $r$ with $k-1$ free knots, i.e., \[
\varphi=\sum_{j=1}^{k}\mathbf{1}_{\left]t_{j-1}, t_{j}\right]}\cdot\pi_{j},\]
where $0=t_{0}<\cdots<t_{k}=1$ and $\pi_{1},\ldots,\pi_{k}\in\Pi_{r}$.
Then, any approximation method $\widehat{X}_{k}$ by splines with
$k-1$ free knots can be thought of as a mapping \[
\widehat{X}_{k}\;:\Omega\longrightarrow\Phi_{k,r},\]
and we denote this class of mappings by $\mathfrak{N}_{k,r}$.

Let $X$ and $\widehat{X}_{k}$ denote the strong solution and an approximate solution on $\left[0,1\right]$, respectively. For the pathwise error we consider the distance
in $L_{\infty}$-norm\[
\bigl\Vert X-\widehat{X}_{k}\bigr\Vert_{L_{\infty}\left[0,1\right]}=\sup_{0\leq t\leq1}\bigl|X\left(t\right)-\widehat{X}_{k}\left(t\right)\bigr|,\]
 and we define the error $e_{q}\bigl(\widehat{X}_{k}\bigr)$ of the approximation
$\widehat{X}_{k}$ by averaging over all trajectories, i.e., \begin{equation}
e_{q}\bigl(\widehat{X}_{k}\bigr)=\left(E^{*}\bigl\Vert X-\widehat{X}_{k}\bigr\Vert_{L_{\infty}\left[0,1\right]}^{q}\right)^{1/q},\qquad  1\leq q<\infty.\label{eq2}\end{equation}
Here we use the outer expectation value $E^{*}$ in order to avoid
cumbersome measurability considerations. The reader is referred to
\cite{WellnerVanderVaart:96} for a detailed study of the outer integral
and expectation.

Furthermore, we define the minimal error
\begin{equation} 
e_{k,q}^{\min}\left(X\right) = \inf \{e_{q}\bigl(\widehat{X}_{k}\bigr) : \widehat{X}_{k} \in \mathfrak{N}_{k,r} \},
\end{equation} 
i.e., the $q$-average $L_{\infty}$-distance of the solution $X$ to the spline space  $\Phi_{k,r}$.

Note, that spline approximation with free knots is a nonlinear approximation problem in the sense that the approximants do not come from linear spaces, but rather from nonlinear manifolds $\Phi_{k,r}$. Nonlinear approximation
for deterministic functions has been extensively discussed, see \cite{Devore:98}
for a survey. In the context of stochastic processes much less is
known and we refer the reader to \cite{CohenD'Ales:97, CohenDaubechiesGuleryuz:02, CreutzigGronbachRitter:07, konPlaskota:05,Slassi:12}. At first in \cite{konPlaskota:05} and thereafter in \cite{CreutzigGronbachRitter:07, Slassi:12}
approximation by splines with free knots is studied, while wavelet
methods are employed in \cite{CohenD'Ales:97,CohenDaubechiesGuleryuz:02}.

In the sequel, for two sequences $\left(a_{k}\right)_{k\in\mathbb{N}}$
and $\left(b_{k}\right)_{k\in\mathbb{N}}$ of positive real numbers
we write $a_{k}\approx b_{k}$ if $\lim_{k\to\infty}a_{k}/b_{k}=1$ 
and $a_{k}\gtrsim b_{k}$ if $\liminf_{k\to\infty}a_{k}/b_{k}\geq1$.
Additionally $a_{k}\asymp b_{k}$ means $C_{1}\leq a_{k}/b_{k}\leq C_{2}$
for all $k\in\mathbb{N}$ and some positive constants $C_{i}$.

Typically, linear splines with fixed knots or with sequential selection of knots are used to approximate the solution of SDEs globally on a time interval. Such approximations are also considered in the present paper. 

For $k\in \mathbb{N}$  we use $\widehat{X}_{k}^{e}$  to denote the piecewise interpolated Euler scheme with constant step-size  $1/k$. In \cite{Faure:01} Faure has given an upper
bound \begin{equation}
e_{q}\left(\widehat{X}_{k}^{e}\right)\leq C\cdot\left(\ln k/k\right)^{1/2}\label{order}\end{equation}
with an unspecified constant $C$. In \cite{Gronbach:02} M\"uller-Gronbach  
has determined the strong asymptotic behaviour of $e_{q}\left(\widehat{X}_{k}^{e}\right)$
with an explicitly given constant, namely\[
e_{q}\left(\widehat{X}_{k}^{e}\right)\approx \frac{C_{q}^{e}}{\sqrt{2}}\cdot\left(\ln k/k\right)^{1/2}\]
with \[
C_{q}^{e}=\left(E\left\Vert \sigma\right\Vert _{L_{\infty}\left[0,1\right]}^{q}\right)^{1/q},\]
where $\left\Vert \sigma\right\Vert _{L_{\infty}\left[0,1\right]}=\sup_{t\in\left[0,1\right]}\left|\sigma\left(t,X\left(t\right)\right)\right|.$

Now, we analyze approximations that are based on a sequential
selection of knots to evaluate $W$, see \cite{Gronbach:02}
for a formal definition. This includes numerical methods with adaptive discretization that reflects the local smoothness of the solution. In \cite{Gronbach:02} M\"uller-Gronbach
shows that a step size proportional to the inverse of the current value of $\left|\sigma\right|^{2}$  leads to an asymptotically optimal method $\widehat{X}_{k}^{a}$, more precisely 
\begin{equation}
e_{q}\left(\widehat{X}_{k}^{a}\right)\approx \frac{C_{q}^{a}}{\sqrt{2}}\cdot\left(\ln k/k\right)^{1/2}\label{klaus1}
\end{equation}
 and \[
C_{q}^{a}=\left(E\left\Vert \sigma\right\Vert _{2}^{q}\right)^{1/q},\]
where $\left\Vert \sigma\right\Vert _{2}=\left(\int_{0}^{1}\left(\sigma\left(t,X\left(t\right)\right)\right)^{2}\,dt\right)^{1/2}.$
Moreover, he establishes strong asymptotic optimality of the sequence
$\widehat{X}_{k}^{a}$, i.e., for every sequence of methods $\widehat{X}_{k}$
that use $k$ sequential observations of $W$ \begin{equation}
e_{q}\left(\widehat{X}_{k}\right)\gtrsim \frac{C_{q}^{a}}{\sqrt{2}}\cdot\left(\ln k/k\right)^{1/2}.\label{klaus2}\end{equation}
Typically $C_{q}^{a}<C_{q}^{e}$ and $C_{q}^{a}>0$, which means that the convergence order $ \left(\ln k/k\right)^{1/2}$ cannot be improved by sequential observation of $W.$

In the present paper we do not impose any restriction on the selection
of the knots, i.e., we assume to have complete information about the
individual paths of $W$ and $X$.

On the other hand we know from Creutzig et al. \cite{CreutzigGronbachRitter:07} that\begin{equation}
e_{k,q}^{\min}\left(X\right)\asymp\left(1/k\right)^{1/2}.\label{klaus3}\end{equation}
Hence the order $1/\sqrt{k}$ is best possible for spline approximation of SDEs with $k-1$ free knots. We add that the same order
of convergence is achieved by the average Kolmogorov widths, see \cite{Creutig:02,Mairove:93,Mairove:96}. 

In the present paper we address the open question how to get an implementable method with error of order $\left(1/k\right)^{1/2}$. Up to now, there are only pure existence results available in literature.
We introduce an approximation method $\widetilde{X}_{\varepsilon}^{**}$ which combines a Milstein scheme with a free knot linear interpolation of the Brownian motion $W$ with guaranteed a priori given accuracy $\varepsilon$. The method $\widetilde{X}_{\varepsilon}^{**}$ progresses from the left to the right and is easy to implement. For the error of $\widetilde{X}_{\varepsilon}^{**}$  we demonstrate the strong asymptotic behaviour with an explicitely given constant, namely
\begin{equation}
 \lim _{\varepsilon \to 0} \left(1/\varepsilon\right) \cdot e_{q}\left(\widetilde{X}_{\varepsilon}^{**}\right) = \left(E\left\Vert \sigma\right\Vert _{L_{\infty}\left[0,1\right]}^{q}\right)^{1/q}\label{approxi5I}
\end{equation}
and with probability one
\begin{equation}
 \lim _{\varepsilon \to 0} \left(1/\varepsilon\right) \cdot \left\Vert X-\widetilde{X}_{\varepsilon}^{**}\right\Vert _{L_{\infty}\left[0,1\right]}  = \left\Vert \sigma\right\Vert _{L_{\infty}\left[0,1\right]},\label{approxi6I}
\end{equation}
where
\begin{equation}
 \tau_{1,1}=\inf \{t>0:\;\sup_{0\leq s\leq t}\bigl|W\left(s\right) - \frac{s}{t}\cdot W\left(t\right)\bigr|>1\}. \label{I3} 
\end{equation}
As a rough measure for the computational cost we use the expectation of the number of free knots $N_{\varepsilon}$ used by $\widetilde{X}_{\varepsilon}^{**}$ pathwise. We show that
\begin{equation} 
\lim_{\varepsilon \to 0}\varepsilon^{2}\cdot E\left( N_{\varepsilon}\right) =  \frac{\pi^{2}}{ 14\cdot\zeta\left(3\right)}\label{I1}
\end{equation}
and with probability one 
\begin{equation}  
\lim_{\varepsilon \to 0}\varepsilon^{2}\cdot N_{\varepsilon} = \frac{\pi^{2}}{ 14\cdot\zeta\left(3\right)}.\label{I2}
\end{equation}  
Here $\zeta\left(\cdot\right)$ denotes the Riemann zeta function.

The structure of the paper is as follows. In Section 2 we present a free knot linear interpolation method $\widetilde {W}_{\varepsilon}$ of the Brownian motion $W$ with guaranteed accuracy $\varepsilon$ on the interval $[0,1]$. Furthermore, we address the following question: What is the computational cost necessary to achieve the error $\varepsilon$? In Section 3 we specify our assumptions regarding the equation (\ref{eq1}). The drift and diffusion coefficients must satisfy Lipschitz conditions and the initial value must have a finite $q$-moment for all $q\geq 1$. Moreover, we introduce the approximation method $\widetilde{X}_{\varepsilon}^{**}$ and determine the strong asymptotic behaviour of the error of $\widetilde{X}_{\varepsilon}^{**}$. The Appendix is devoted to the analysis of the approximation error of the Milstein method with random step size, which is useful for our main result.  
\section{Free knot linear interpolation of the Brownian motion} 
In this Section we introduce a free knot linear interpolation  $\widetilde{W}_{\varepsilon}$  of $W$ with guaranteed accuracy $\varepsilon$. The way of choosing the knots is motivated by the method of free knot spline approximation introduced in \cite{CreutzigGronbachRitter:07}. Note that piecewise linear interpolation of the Brownian motion at free knots with pathwise guaranteed accuracy has also been used in \cite{Grill:87} to study rates of convergence in the functional law of the iterated logarithm.
%approximation method $\widetilde {X}_{\varepsilon}^{**}$ that is based on a linear interpolation of the Brownian motion $W$ with at most $\lfloor 1/\varepsilon^{2}\rfloor$ number of free knots. 

Let $ \widetilde{W}_{s,t}$ denote the linear interpolation of $W$ at $s$ and $t$.\\ Given error bound $\varepsilon >0,$ we define a sequence of stopping times by $\tau_{0,\varepsilon}=0$ and for $j\in \mathbb{N}$
\begin{equation}
\tau_{j,\varepsilon}=\inf\left\{ t>\tau_{j-1,\varepsilon}\;\mid\;\left\Vert W-\widetilde{W}_{\tau_{j-1,\varepsilon},\;t}\right\Vert _{L_{\infty}\left[\tau_{j-1,\varepsilon},\; t\right]}>\varepsilon\right\}.\label{stop1}
\end{equation}
For $j\in\mathbb{N}$ we define\[ 
\xi_{j,\varepsilon}=\tau_{j,\varepsilon}-\tau_{j-1,\varepsilon}.\]
These random variables yield the lengths of consecutive maximal subintervals 
that permit piecewise linear interpolation with error
at most $\varepsilon$. For every $\varepsilon>0$ the random variables
$\xi_{j,\varepsilon}$ form an i.i.d. sequence with \begin{equation}
\xi_{j,\varepsilon}\stackrel{d}{=}\varepsilon^{2}\cdot\tau_{1,1}\quad\mathrm{and}\quad E\left(\left(\tau_{1,1}\right)^{m}\right)<\infty \label{asymptotic9}
\end{equation}
for every $m\in\mathbb{N}$, see \cite{Ehlenz:08, Grill:87}.

The following Lemma yields the distribution of the stopping time $\tau_{1,1}$. 
\begin{Lem} \label{Verteilung}
\begin{enumerate}
\item For all $x>0$ we have $$ \mathbb{P} \left(\tau_{1,1}< x \right) = 1-k\left(1/\sqrt{x}\right),$$ where 
 $$ K\left(x\right)= \left\{ \begin{array}{ll}
\sum_{i=-\infty}^{\infty}\left(-1\right)^{i} e^{-2i^{2}\cdot x^{2}} &  \mathrm{for}\; x>0\\
 0 & \mathrm{for}\; x\leq 0 \end{array}\right. $$ is the Kolmogorov distribution function. 
\item We have $$ E\left(\tau_{1,1}\right)=  \left(14\cdot\zeta\left(3\right)\right)/\pi^{2},$$ where
 $$ \zeta\left(s\right) =\sum_{n=1}^{\infty}\frac{1}{n^{s}}$$ is the Riemann zeta function.
\end{enumerate}
\end{Lem} 
\begin{proof}
\textbf{ad(1)} At first, note that 
$$ \tau_{1,1}=\inf \{t>0:\;\sup_{0\leq s\leq t}\bigl|W\left(s\right) - \frac{s}{t}\cdot W\left(t\right)\bigr|>1 \}.$$
Then, we have for all $x>0$
\begin{eqnarray*}
 \mathbb{P}\left(\tau_{1,1} <x\right) &=&  \mathbb{P}\left( \sup_{0\leq s\leq x}\bigl|W\left(s\right) - \frac{s}{x}\cdot W\left(x\right)\bigr|>1\right)\\
                                           &=& \mathbb{P}\left( \sup_{0\leq s\leq 1}\bigl|W\left(s\cdot x\right) - s\cdot W\left(x\right)\bigr|>1\right)\\
                                            &=& \mathbb{P}\left( \sup_{0\leq s\leq 1}\Bigl|\frac{W\left(s\cdot x\right)}{\sqrt{x}} - \frac{s\cdot W\left(x\right)}{\sqrt{x}}\Bigr|>\frac{1}{\sqrt{x}}\right)\\
                                           &=& 1- \mathbb{P}\left( \sup_{0\leq s\leq 1}\Bigl|\frac{W\left(s\cdot x\right)}{\sqrt{x}} - \frac{s\cdot W\left(x\right)}{\sqrt{x}}\Bigr|\leq\frac{1}{\sqrt{x}}\right).
\end{eqnarray*}
From this follows the first assertion in Lemma \ref{Verteilung}, see \cite{Billingsley:68}.\\
\textbf{ad(2)}  At first, standard relation for theta functions yields $$ K\left(x\right) = \frac{\sqrt{2\pi}}{x}\sum_{i=1}^{\infty}e^{-\left(2i-1\right)^{2}\pi^{2}/\left(8x^{2}\right)}.$$
 Then, we have
\begin{eqnarray*}
E\left(\tau_{1,1}\right)   &=& \int_{0}^{\infty} \mathbb{P}\left(\tau_{1,1}>x\right)\,dx\\
                                &=& \int_{0}^{\infty}K\left(1/\sqrt{x}\right)\,dx \\
                                &=&\sqrt{2\pi}\cdot \int_{0}^{\infty}\sqrt{x}\sum_{i=1}^{\infty}e^{-\left(\left(2i-1\right)^{2}\pi^{2}/8\right)\cdot x}\,dx\\
                                &=&\sqrt{2\pi}\cdot \sum_{i=1}^{\infty}\left(\frac{8}{\left(2i-1\right)^{2}\pi^{2}}\right)^{3/2}\cdot \Gamma\left(\frac{1}{2}+1\right)\\
                                &=& \frac{14}{\pi^{2}}\cdot\zeta\left(3\right). 
\end{eqnarray*}
This completes the proof.
\end{proof}
%Here is the density curve of the stopping time $\tau^{sec}_{1,1}$:
\begin{figure}[h]
\begin{center}
%\hspace*{-4cm} 
The probability density curve of the stopping time $\tau_{1,1}$
\rotatebox{0}{\begin{minipage}{5cm} 
\scalebox{.3}{\includegraphics{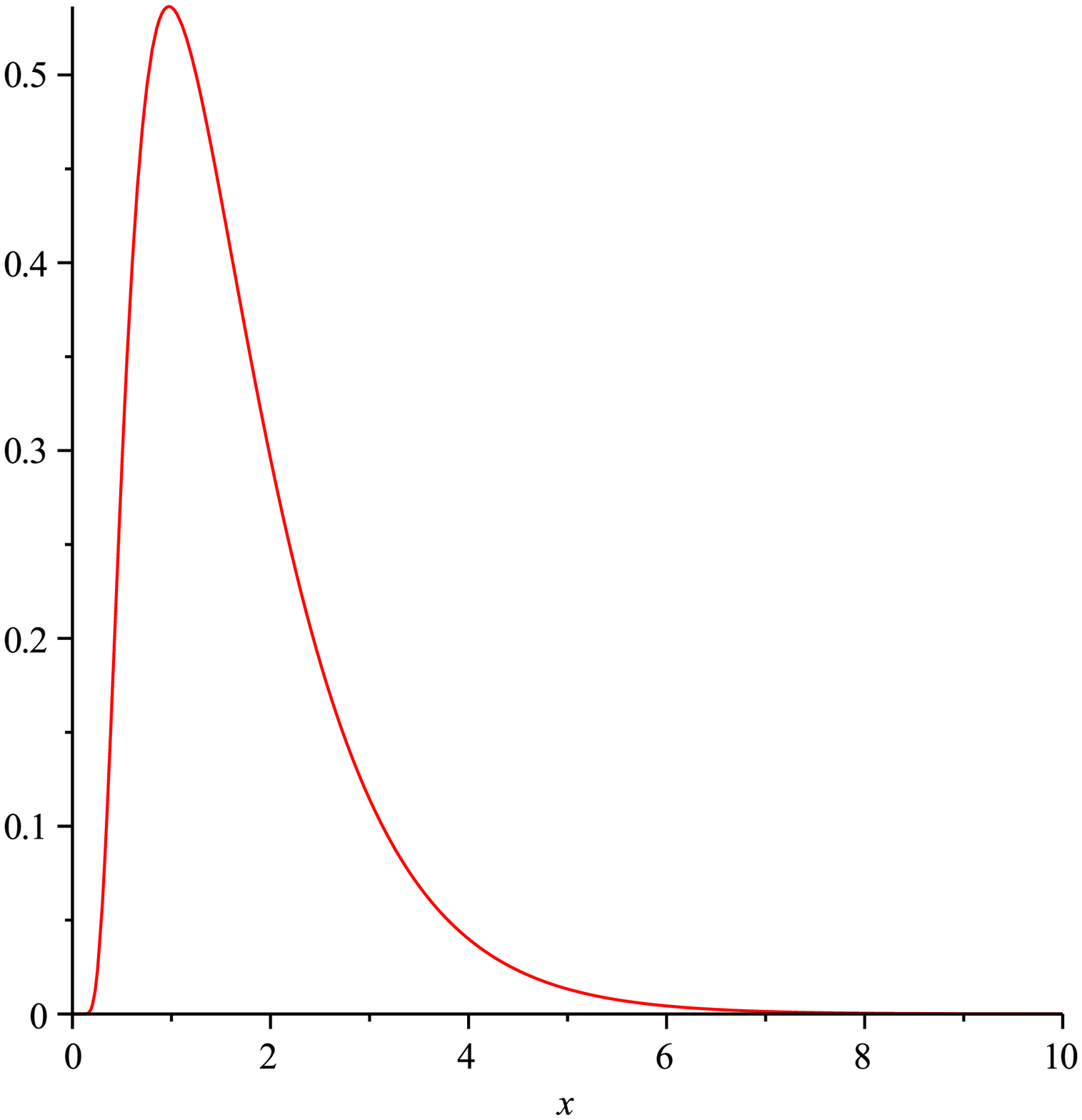}}
\end{minipage}}
%\hspace*{80pt}
\end{center}
\end{figure}
\newpage
We define 
\begin{equation} 
N_{\varepsilon}=N_{\varepsilon}\left(W\right) =\sup \{ n\geq 0 \;:\; S_{n}<1 \},\label{stop2}
\end{equation}
where
$$ S_{0}=0 \qquad \mathrm{and}\qquad S_{n}=\sum_{j=1}^{n}\xi_{j,\varepsilon}, \quad n \in \mathbb{N}.$$
Note, that we have $$E\left(N_{\varepsilon}\right)=\sum_{n=1}^{\infty}\mathbb{P}\left(S_{n}\leq 1 \right)$$
and hence $$ E\left(N_{\varepsilon}\right)<\infty.$$
A free knot linear interpolation of $W$ with guaranteed accuracy $\varepsilon$ on $\left[0,1\right] $ is given by $$\widetilde{W}_{\varepsilon}=\sum_{j=1}^{N_{\varepsilon}+1}\mathbf{1}_{\left]\tau_{j-1,\varepsilon}, \tau_{j,\varepsilon}\right]}\cdot\widetilde{W}_{\tau_{j-1,\varepsilon},\tau_{j,\varepsilon}}.$$ We have 
\begin{equation}
\left\Vert W-\widetilde{W}_{\varepsilon}\right\Vert _{L_{\infty}\left[\tau_{j-1,\varepsilon},\tau_{j,\varepsilon}\right]} = \varepsilon \label{interpo3}\end{equation}
for $j=1,\ldots,N_{\varepsilon}+1$. Note, that $ N_{\varepsilon}$ is the number of free knots on $\left[0,1\right]$ used by $\widetilde{W}_{\varepsilon}$.
\begin{Prop}\label{LemN} For the random variable $ N_{\varepsilon}$ we have:
\begin{enumerate}
\item  $ N_{\varepsilon} \to \infty$, as $\varepsilon \to 0$  a.s.
\item $\lim_{\varepsilon \to 0}\varepsilon^{2}\cdot E\left(N_{\varepsilon}\right) = 1/E\left(\tau_{1,1}\right)$. 
\item $\lim_{\varepsilon \to 0}\varepsilon^{2}\cdot N_{\varepsilon} = 1/E\left(\tau_{1,1}\right)$  a.s. 
\end{enumerate}
\end{Prop}
\begin{proof}
\textbf{ad(1)} For all $n\in \mathbb{N}$ we have $$ \mathbb{P}\left(N_{\varepsilon}<n\right)= 1- \mathbb{P}\left(S_{n}\leq 1\right).$$ We use (\ref{asymptotic9}) and  (\ref{asymptotic10}) to obtain
\begin{eqnarray*} 
\mathbb{P}\left(S_{n}\leq 1\right) &\geq & \left(\mathbb{P}\left(\xi_{1,\varepsilon}\leq \frac{1}{n}\right)\right)^{n}\\
                                  & \geq & \left(1-\exp\left(-\frac{C}{n\cdot\varepsilon^{2}}\right)\right)^{n}.
\end{eqnarray*}
This yields $$ \lim_{\varepsilon \to 0}\mathbb{P}\left(N_{\varepsilon}<n \right) = 0 $$
for all $n\in \mathbb{N}$, so $N_{\varepsilon}\to \infty$, as $\varepsilon \to 0$ in probability. Since $N_{\varepsilon}$ is increasing, holds  $N_{\varepsilon}\to \infty$, as $\varepsilon \to 0$ almost surely.\\
\textbf{ad(2)} Note, that $N_{\varepsilon} +1$ is a stopping time.Therefore we use the Wald's equation to obtain $$E\left(S_{N_{\varepsilon}+1}\right)= E\left(N_{\varepsilon}+1\right)\cdot E\left(\tau_{1,\varepsilon}\right)$$ and hence  
\begin{equation} 
\frac{1}{E\left(\tau_{1,\varepsilon}\right)} -1 \leq E\left(N_{\varepsilon}\right) \leq  \frac{1+ E\left(\xi_{N_{\varepsilon}+1,\,\varepsilon}\right)}{E\left(\tau_{1,\varepsilon}\right)} -1.\label{asymptotic11}  
\end{equation}
Since $\mathbf{1}_{\{N_{\varepsilon} =n\}}$ and $\xi_{n+1,\varepsilon}$ are independent for all $n\geq 0$, we get
$$ E\left(\xi_{N_{\varepsilon}+1,\,\varepsilon}\right) = E\left(\tau_{1,\varepsilon}\right)$$ and hence from (\ref{asymptotic11}) $$\lim_{\varepsilon \to 0}\varepsilon^{2}\cdot E\left(N_{\varepsilon}\right) = 1/E\left(\tau_{1,1}\right)$$ by (\ref{asymptotic9}).\\
\textbf{ad(3)} Due to $S_{N_{\varepsilon}}<1 \leq S_{N_{\varepsilon}+1}$ we have 
$$ \frac{\left(1/\varepsilon^{2}\right)\cdot S_{N_{\varepsilon}}}{N_{\varepsilon}}< \frac{1/\varepsilon^{2}}{N_{\varepsilon}}\leq \frac{\left(1/\varepsilon^{2}\right)\cdot S_{N_{\varepsilon}+1}}{N_{\varepsilon}+1} \cdot\frac{N_{\varepsilon}+1}{N_{\varepsilon}}.$$
Since  $ N_{\varepsilon} \to \infty$ a.s. as $\varepsilon \to 0$ we get by the strong law of large numbers that
$$ \lim_{\varepsilon \to 0} \frac{1}{\varepsilon^{2}\cdot N_{\varepsilon}} = E\left(\tau_{1,1}\right)\quad \mathrm{a.s.}$$
\end{proof}
\begin{Rem}\label{remark}
For $j\in\mathbb{N}$ we define $$
\Lambda_{j,\varepsilon}=\frac{W\left(\tau_{j,\varepsilon}\right)-W\left(\tau_{j-1,\varepsilon}\right)}{\left(\tau_{j,\varepsilon}-\tau_{j-1,\varepsilon}\right)^{1/2}}.$$ From \cite{Ehlenz:08} we know that for every $\varepsilon >0$  the random variables $\Lambda_{j,\varepsilon}$ form an i.i.d. sequence with
$$ \Lambda_{j,\varepsilon} \stackrel{d}{=} N\left(0,1\right).$$
Moreover, $\left(\Lambda_{j,\varepsilon}\right)_{j\in \mathbb{N}}$ and $\left(\xi_{j,\varepsilon}\right)_{j\in \mathbb{N}}$ are independent.  See \cite{Grill:87} too.\\
Using these facts, Lemma \ref{Verteilung}, (\ref{asymptotic9}) and (\ref{interpo3}) it is sufficient to generate realizations of the stopping time  $ \tau_{1,1}$ and standard normally distributed random variables to simulate the free knot linear interpolation of Brownian paths with guaranteed accuracy $\varepsilon$ on $\left[0,1\right]$.
\end{Rem}
\section{The main result}
In this Section we present the asymptotic analysis for the approximation method $\widetilde{X}_{\varepsilon}^{**}$. The method  $\widetilde{X}_{\varepsilon}^{**}$ combines a Milstein scheme with the free knot linear interpolation of the Brownian motion introduced in Section 2.

Throughout this paper we assume that the drift and diffusion coefficient $$a,\:\sigma:\left[0,\infty \right)\times\mathbb{R}\to\mathbb{R}$$
and the initial value $X\left(0\right)$ have following properties.
\begin{itemize}
\item $\mathrm{(A)}$ Both, $a$ and $\sigma$ are differentiable with respect 
to the state variable. Moreover, there exists a constant $K>0$ such
that $f=a$ and $f=\sigma$ satisfy\begin{eqnarray*}
\left|f\left(t,x\right)-f\left(t,y\right)\right| & \leq & K\cdot\left|x-y\right|,\\
\left|f\left(s,x\right)-f\left(t,x\right)\right| & \leq & K\cdot\left(1+\left|x\right|\right)\cdot\left|s-t\right|,\\
\left|f^{\left(0,1\right)}\left(t,x\right)-f^{\left(0,1\right)}\left(t,y\right)\right| & \leq & K\cdot\left|x-y\right|\end{eqnarray*}
for all $s,\; t\in\left[0,\infty \right)$ and $x,\; y\in\mathbb{R}.$ 
\item $\mathrm{(B)}$ The initial value $X\left(0\right)$ is independent
of $W$ and \[
E\left(\left|X\left(0\right)\right|^{q}\right)<\infty\qquad\mathrm{for\; all}\quad q\geq1.\]
\end{itemize}
Note, that $\mathrm{(A)}$ yields the linear growth condition, i.e., there exists a constant $c > 0$ such that \begin{equation}
\left|f\left(t,x\right)\right|\leq c\cdot\left(1+\left|x\right|\right)\label{growth}\end{equation} for all $ t \in \left[0,\infty\right)$ and $x\in\mathbb{R}.$ 
Moreover, $f^{\left(0,1\right)}$ is bounded and \[
\left|f\left(t,x\right)-f\left(t,y\right)-f^{\left(0,1\right)}\left(t,y\right)\left(x-y\right)\right|\leq c\cdot\left(x-y\right)^{2}.\]
Given the above properties $\mathrm{(A)}$ and $\mathrm{(B)}$, a pathwise unique strong solution of equation
(\ref{eq1}) on $\left[0,1\right]$ with initial value $X\left(0\right)$ exists. In particular
the conditions assure that\begin{equation}
E\left(\left\Vert X\right\Vert _{L_{\infty}\left[0,1\right]}^{q}\right)<\infty\qquad\mathrm{for\; all}\; q\geq1.\label{eqEX}\end{equation}

\emph{Construction of the approximation method} $\widetilde{X}_{\varepsilon}^{**}$. Put  
\begin{equation}
\tau_{\ell}=\tau_{\ell,\varepsilon}\label{discre2}
\end{equation} 
for $\ell = 0,\dots,N_{\varepsilon}$ and $\tau_{N_{\varepsilon}+1}=1$. We take the Milstein scheme to compute an approximation to $X$ at
the discrete points $\tau_{\ell}$. This scheme is defined by $$\breve{X}_{\varepsilon}\left(\tau_{0}\right) =X\left(0\right)$$
and
\begin{eqnarray}
\breve{X}_{\varepsilon}\left(\tau_{\ell}\right)  &= &  \breve{X}_{\varepsilon}\left(\tau_{\ell-1}\right)+ a\left(\tau_{\ell-1},\breve{X}_{\varepsilon}\left(\tau_{\ell-1}\right)\right)\cdot\left(\tau_{\ell}-\tau_{\ell-1}\right) \nonumber \\
& & +\sigma\left(\tau_{\ell-1},\breve{X}_{\varepsilon}\left(\tau_{\ell-1}\right)\right)\cdot\left(W\left(\tau_{\ell}\right)-W\left(\tau_{\ell-1}\right)\right)  \label{milstein}\\
& & + \frac{1}{2}\cdot\left(\sigma\cdot\sigma^{\left(0,1\right)}\right)\left(\tau_{\ell-1},\breve{X}_{\varepsilon}\left(\tau_{\ell-1}\right)\right)\cdot\left(\left(W\left(\tau_{\ell}\right)-W\left(\tau_{\ell-1}\right)\right)^{2}-\left(\tau_{\ell}-\tau_{\ell-1}\right)\right),\nonumber
\end{eqnarray} 
where $\sigma^{\left(0,1\right)}$ denotes the partial derivate of
$\sigma$ with respect to the second or state variable.\\
Now, method  $\widetilde{X}_{\varepsilon}^{**}$ is given by\[
\widetilde{X}_{\varepsilon}^{**}\left(\tau_{0}\right)=X\left(0\right)\]
and for $t\in\left]\tau_{\ell-1},\tau_{\ell}\right]$
 \begin{eqnarray}
\widetilde{X}_{\varepsilon}^{**}\left(t\right) & = &  \breve{X}_{\varepsilon}\left(\tau_{\ell-1}\right)+ a\left(\tau_{\ell-1},\breve{X}_{\varepsilon}\left(\tau_{\ell-1}\right)\right)\cdot\left(t-\tau_{\ell-1}\right) \nonumber \\
                                       & & +\sigma\left(\tau_{\ell-1},\breve{X}_{\varepsilon}\left(\tau_{\ell-1}\right)\right)\cdot\left(\widetilde{W}_{\varepsilon}\left(t\right) -W\left(\tau_{\ell-1}\right)\right).  \label{linearsplineinter1}
\end{eqnarray}
 Observe that method $\widetilde{X}_{\varepsilon}^{**}$ uses $N_{\varepsilon}$ free knots pathwise and its computational cost is given by $E\left(N_{\varepsilon}\right)$. On the basis of above preparations the main result can now be stated.
\begin{Thm}
\label{theorem1-4} The method $\widetilde{X}_{\varepsilon}^{**}$ satisfies
\begin{equation}
 \lim _{\varepsilon \to 0} \left(1/\varepsilon \right) \cdot e_{q}\left(\widetilde{X}_{\varepsilon}^{**}\right) =\left(E\left\Vert \sigma\right\Vert _{L_{\infty}\left[0,1\right]}^{q}\right)^{1/q}\label{approxi3}
\end{equation}
and 
\begin{equation}
 \lim _{\varepsilon \to 0} \left(1/\varepsilon \right) \cdot \left\Vert X-\widetilde{X}_{\varepsilon}^{**}\right\Vert _{L_{\infty}\left[0,1\right]}  = \left\Vert \sigma\right\Vert _{L_{\infty}\left[0,1\right]}\quad \mathrm{a.s.}\label{approxi4}
\end{equation}
for all $q\geq 1$ and every equation (\ref{eq1}).  
\end{Thm}
\begin{proof}
In order to prove the main results given in Theorem \ref{theorem1-4},
we introduce processes $X^{M}_{\varepsilon}$ as follows. For $\varepsilon >0$
let \[
0=\tau_{0}<\tau_{1}<\cdots<\tau_{N_{\varepsilon}+1}=1\]
be the discretization (\ref{discre2}) of $\left[0,1\right]$. Now, the
process $X^{M}_{\varepsilon}$ is given by $X^{M}_{\varepsilon}\left(0\right)=X\left(0\right)$
and for $t\in\left[\tau_{\ell-1},\tau_{\ell}\right]$\begin{eqnarray}
X^{M}_{\varepsilon}\left(t\right) & = & X^{M}_{\varepsilon}\left(\tau_{\ell-1}\right)+a\left(\tau_{\ell-1},X^{M}_{\varepsilon}\left(\tau_{\ell-1}\right)\right)\cdot\left(t-\tau_{\ell-1}\right)\nonumber \\
 &  & +\sigma\left(\tau_{\ell-1},X^{M}_{\varepsilon}\left(\tau_{\ell-1}\right)\right)\cdot\left(W\left(t\right)-W\left(\tau_{\ell-1}\right)\right)\label{milst} \\ 
 &  & +1/2\cdot\left(\sigma\cdot\sigma^{\left(0,1\right)}\right)\left(\tau_{\ell-1},X^{M}_{\varepsilon}\left(\tau_{\ell-1}\right)\right)\cdot\left(\left(W\left(t\right)-W\left(\tau_{\ell-1}\right)\right)^{2}-\left(t-\tau_{\ell-1}\right)\right).
\nonumber 
\end{eqnarray}
Note, that at the discretization points $\tau_{\ell}$ the processes $X^{M}_{\varepsilon}$ coincides with the Milstein scheme (\ref{milstein}).
Instead of estimating $X-\widetilde{X}_{\varepsilon}^{**}$
directly, we consider $X-X^{M}_{\varepsilon}$ as well as $X^{M}_{\varepsilon}-\widetilde{X}_{\varepsilon}^{**}$
separately. From Theorem \ref{theorem} in Appendix it follows that\[
\lim_{\varepsilon\to 0}\left(1/\varepsilon\right)\cdot\left(E\left\Vert X-X^{M}_{\varepsilon}\right\Vert _{L_{\infty}\left[0,1\right]}^{q}\right)^{1/q}=0,\]
and so $\left(E\left\Vert X^{M}_{\varepsilon}-\widetilde{X}_{\varepsilon}^{**}\right\Vert _{L_{\infty}\left[0,1\right]}^{q}\right)^{1/q}$
is  asymptotically the dominating term.

From now on let $C$ denote unspecified positive constants, which
only depend on the constant $K$ from condition $\mathrm{(A)}$ as
well as on $a\left(0,0\right),\;\sigma\left(0,0\right)$ and $E\bigl|X\left(0\right)\bigr|^{q}$.

\textbf{Proof of the upper bound in (\ref{approxi3}).} Put $U_{\ell}=\left(\tau_{\ell},X^{M}_{\varepsilon}\left(\tau_{\ell}\right)\right)$. Then for $t\in\left]\tau_{\ell-1},\tau_{\ell}\right]$ we have
\begin{eqnarray*}
\left|X^{M}_{\varepsilon}\left(t\right)-\widetilde{X}_{\varepsilon}^{**}\left(t\right)\right| &\leq &\left|\sigma\left(U_{\ell-1}\right)\cdot \left( W\left(t\right)-\widetilde{W}_{\varepsilon}\left(t\right)\right)\right| +\\
 &  &  \left|1/2\cdot\left(\sigma\cdot\sigma^{\left(0,1\right)}\right)\left(U_{\ell-1}\right)\cdot\left( \left(W\left(t\right) - W\left(\tau_{\ell-1}\right)\right)^{2}-\left(t-\tau_{\ell-1}\right)\right)\right|.
\end{eqnarray*}
Thus, by (\ref{interpo3})
\begin{eqnarray*}
\lefteqn{\left\Vert X^{M}_{\varepsilon}-\widetilde{X}_{\varepsilon}^{**}\right\Vert _{L_{\infty}\left[0,1\right]} \leq  \max_{1\leq\ell\leq N_{\varepsilon}+1}\left|\sigma\left(U_{\ell-1}\right)\right|\cdot \varepsilon + }\\
& & \max_{1\leq\ell\leq N_{\varepsilon}+1}\left(\left|1/2\cdot\left(\sigma\cdot\sigma^{\left(0,1\right)}\right)\left(U_{\ell-1}\right)\right|\cdot\sup_{\tau_{\ell-1} < t\leq \tau_{\ell}}\left|\left(W\left(t\right)-W\left(\tau_{\ell-1}\right) \right)^{2}-\left(t-\tau_{\ell-1}\right)\right|\right).
\end{eqnarray*} 
Minkowski's inequality yields\[
\left(E^{*}\left\Vert X^{M}_{\varepsilon}\left(t\right)-\widetilde{X}_{\varepsilon}^{**}\left(t\right)\right\Vert _{L_{\infty}\left[0,1\right]}^{q}\right)^{1/q}\leq I\left(\varepsilon\right)+J\left(\varepsilon\right),\]
where 
$$ I\left(\varepsilon\right)=\left(E\max_{1\leq\ell\leq N_{\varepsilon}+1}\left|\sigma\left(U_{\ell-1}\right)\right|^{q}\cdot\varepsilon^{q}\right)^{1/q}$$
and  
\begin{eqnarray*}
\lefteqn{J\left(\varepsilon\right)= }\\
& & \left(E\max_{1\leq\ell\leq N_{\varepsilon}+1}\left(\left|\frac{1}{2}\left(\sigma\cdot\sigma^{\left(0,1\right)}\right)\left(U_{\ell-1}\right)\right|\cdot\sup_{\tau_{\ell-1}< t\leq \tau_{\ell}}\left|\left(W\left(t\right) - W\left(\tau_{\ell-1}\right) \right)^{2}-\left(t-\tau_{\ell-1}\right)\right|\right)^{q}\right)^{1/q}.
\end{eqnarray*}
First, from the H\"older's inequality and the boundedness of $\sigma^{\left(0,1\right)}$ as well as the linear growth condition (\ref{growth}) it follows that 
\begin{eqnarray*}
J\left(\varepsilon\right)&\leq & C\cdot \Biggl[\left(E\left(1+\max_{1\leq\ell\leq N_{\varepsilon}+1}\left|X^{M}_{\varepsilon}\left(\tau_{\ell-1}\right) \right|^{2q}\right)\right)^{1/2q} \\
 & & \times \left( E\max_{1\leq\ell\leq N_{\varepsilon}+1}\sup_{\tau_{\ell-1}< t\leq \tau_{\ell}}\left|\left(W\left(t\right) -  W\left(\tau_{\ell-1}\right) \right)^{2}-\left(t-\tau_{\ell-1}\right)\right|^{2q}\right)^{1/2q}\Biggl].
\end{eqnarray*}
From Lemma \ref{lema1-1} in the Appendix we get on the one hand
\begin{align}
\left(E\left(1+\max_{1\leq\ell\leq N_{\varepsilon}+1}\left|X^{M}_{\varepsilon}\left(\tau_{\ell-1}\right) \right|^{2q}\right)\right)^{1/2q} & \leq \;
C\cdot \left(1+ \left(E\left\Vert X^{M}_{\varepsilon}\right\Vert _{L_{\infty}\left[0,1\right]}^{2q}\right)^{1/2q} \right) \leq  C.\label{S1}
\end{align}
On the other hand, according to the H\"older continuity of  $W$ there exists for every $\kappa \in \left(0,1/2\right)$ a nonnegative random variable $\eta_{\kappa}$ with $E\left(\left|\eta_{\kappa}\right|^{m}\right) <\infty$ for all $1\leq m<\infty$ so that we have  almost surely 
$$  \left|W\left(s+\tau_{\ell-1}\right) -  W\left(\tau_{\ell-1}\right)\right|\leq \eta_{\kappa} \cdot\left( \tau_{\ell}-\tau_{\ell-1}\right)^{\left(1-\kappa\right)/2}$$ for all $ s\in \left(0,\, \tau_{\ell}-\tau_{\ell-1}\right]$  and  $\ell \in  \{1,\ldots,N_{\varepsilon}+1\}$. From this and Lemma \ref{lema1}, it follows that
\begin{eqnarray}
 \lefteqn{\left( E\max_{1\leq\ell\leq N_{\varepsilon}+1}\sup_{\tau_{\ell-1}< t\leq \tau_{\ell}}\left|\left(W\left(t\right) - W\left(\tau_{\ell-1}\right) \right)^{2}-\left(t-\tau_{\ell-1}\right)\right|^{2q}\right)^{1/2q} }\nonumber\\
& & \leq C\cdot \left(\left(E \max_{1\leq \ell \leq k}\left(\tau_{\ell}-\tau_{\ell-1}\right)^{2q\cdot\left(1-\kappa\right)}\right)^{1/2q} + \left(E \max_{1\leq \ell \leq k}\left(\tau_{\ell}-\tau_{\ell-1}\right)^{2q}\right)^{1/2q}\right)\nonumber\\
& &\leq  C\cdot \left(\left(\varepsilon\cdot\ln\left(1/\varepsilon\right)\right)^{2\left(1-\kappa\right)}   + \left(\varepsilon\cdot \ln\left(1/\varepsilon\right)\right)^{2} \right). \label{S2}
\end{eqnarray}
Hence from  (\ref{S1}) and (\ref{S2}) we obtain 
\begin{equation}
\limsup_{\varepsilon \to0}\left(1/\varepsilon\right)\cdot J\left(\varepsilon\right)=0.\label{eqJk1}\end{equation}
By  Minkowski's inequality we have
\begin{eqnarray*}
\lefteqn{I\left(\varepsilon\right)\leq \varepsilon \cdot\Bigg( \left(E\max_{1\leq\ell\leq N_{\varepsilon}+1}\left|\sigma\left(U_{\ell-1}\right)-\sigma\left(\tau_{\ell-1},X\left(\tau_{\ell-1}\right)\right)\right|^{q}\right)^{1/q} +} \\
& &\left(E\max_{1\leq\ell\leq N_{\varepsilon}+1}\left|\sigma\left(\tau_{\ell-1},X\left(\tau_{\ell-1}\right)\right)\right|^{q}\right)^{1/q} \Bigg).
\end{eqnarray*}
On the one hand we use the Lipschitz conditions $\mathrm{(A)}$ and Theorem \ref{theorem} in the Appendix to obtain
\begin{equation}
 \left(E\max_{1\leq\ell\leq N_{\varepsilon}+1}\left|\sigma\left(U_{\ell-1}\right)-\sigma\left(\tau_{\ell-1},X\left(\tau_{\ell-1}\right)\right)\right|^{q}\right)^{1/q}
  \leq  C\cdot\left(\varepsilon\cdot \ln\left(1/\varepsilon\right)\right)^{2}\label{S0}
\end{equation}
%\begin{eqnarray*}
% \left(E\max_{0\leq\ell\leq %k-1}\left|\sigma\left(U_{\ell}\right)-\sigma\left(\tau_{\ell},X\left(\tau_{\ell}\right)\r%ight)\right|^{2q}\cdot E\gamma_{k}^{2q}\right)^{1/2q}\qquad\quad \quad \quad\qquad & & \\
 % \lesssim  C\cdot\left(1/\left(k\cdot\sqrt{k}\right)\right) \qquad\qquad & &
%\end{eqnarray*} 
and hence \begin{equation}
\limsup_{\varepsilon \to 0} \left(E\max_{1\leq\ell\leq N_{\varepsilon}+1}\left|\sigma\left(U_{\ell-1}\right)-\sigma\left(\tau_{\ell-1},X\left(\tau_{\ell-1}\right)\right)\right|^{q}\right)^{1/q}  = 0.
\label{eqJk2}
\end{equation}
On the other hand we have \[
\mathbb{P}\left(\lim_{\varepsilon\to 0}\max_{1\leq\ell\leq N_{\varepsilon}+1}\left|\sigma\left(\tau_{\ell-1},X\left(\tau_{\ell-1}\right)\right)\right|^{q}=\left\Vert \sigma\right\Vert _{L_{\infty}\left[0,1\right]}^{q}\right)=1,\]
and from (\ref{growth}) and (\ref{eqEX}) \[
\mathbb{P}\left(\left|\sigma\left(\tau_{\ell-1},X\left(\tau_{\ell-1}\right)\right)\right|^{q}\leq C\cdot\left(1+\left\Vert X\right\Vert _{L_{\infty}\left[0,1\right]}^{q}\right)\right)=1,\]
and\[
E\left(1+\left\Vert X\right\Vert _{L_{\infty}\left[0,1\right]}^{q}\right)<\infty.\]
Hence Lebesgue's theorem yields
\begin{equation}
\lim_{\varepsilon\to 0} \left(E\max_{1\leq\ell\leq N_{\varepsilon}+1}\left|\sigma\left(\tau_{\ell-1},X\left(\tau_{\ell-1}\right)\right)\right|^{q}\right)^{1/q}
= \left( E\left\Vert \sigma\right\Vert _{L_{\infty}\left[0,1\right]}^{q}\right)^{1/q}. \label{eqJk3}
\end{equation}
From this and (\ref{eqJk2}) we obtain\begin{equation}
\limsup_{\varepsilon\to 0}\left(1/\varepsilon\right)\cdot I\left(\varepsilon\right)\leq \left(E\left\Vert \sigma\right\Vert _{L_{\infty}\left[0,1\right]}^{q}\right)^{1/q}.\label{eqJk4}\end{equation}
Now, the upper bound in (\ref{approxi3})  
follows immediately from (\ref{eqJk1}) and (\ref{eqJk4}).

\textbf{Proof of the lower bound in (\ref{approxi3})}. We use (\ref{S1}), (\ref{S2}) and Theorem \ref{theorem} to obtain for every $\kappa \in \left(0,1/2\right)$ \[
\left(E\left\Vert X-\widetilde{X}_{\varepsilon}^{**}\right\Vert _{\infty}^{q}\right)^{1/q}\geq I\left(\varepsilon\right) - C\cdot \left(\varepsilon\cdot\ln\left(1/\varepsilon\right)\right)^{2\left(1-\kappa\right)} .\]
From this and (\ref{S0}) we get  $$\left(E\left\Vert X-\widetilde{X}_{\varepsilon}^{**}\right\Vert _{\infty}^{q}\right)^{1/q}\geq \left(E\max_{1\leq\ell\leq N_{\varepsilon}+1}\left|\sigma\left(\tau_{\ell-1},X\left(\tau_{\ell-1}\right)\right)\right|^{q} \cdot \varepsilon^{q}\right)^{1/q} - C\cdot\left(\varepsilon\cdot\ln\left(1/\varepsilon\right)\right)^{2\left(1-\kappa\right)}.$$ 
Hence Fatou's Lemma implies $$\liminf_{\varepsilon \to 0}\left(1/\varepsilon\right)\cdot\left(E\left\Vert X-\widetilde{X}_{\varepsilon}^{**}\right\Vert _{\infty}^{q}\right)^{1/q}\geq 
\left(E\left\Vert \sigma\right\Vert _{L_{\infty}\left[0,1\right]}^{q}\right)^{1/q},$$ which completes the proof of (\ref{approxi3}).
 
By using Corollary \ref{kor1} in Appendix and the same arguments as in the proof of (\ref{approxi3}) follows the assertion (\ref{approxi4}).
\end{proof}
\begin{Cor} From Proposition \ref{LemN} and Theorem \ref{theorem1-4} we deduce that\begin{equation}
 \lim _{\varepsilon \to 0} \sqrt{E\left(N_{\varepsilon}\right)} \cdot e_{q}\left(\widetilde{X}_{\varepsilon}^{**}\right) = \left(E\left(\tau_{1,1}\right)\right)^{-1/2}\cdot \left(E\left\Vert \sigma\right\Vert _{L_{\infty}\left[0,1\right]}^{q}\right)^{1/q}\label{approxi5}
\end{equation}
and
\begin{equation}
 \lim _{\varepsilon \to 0} \sqrt{N_{\varepsilon}} \cdot \left\Vert X-\widetilde{X}_{\varepsilon}^{**}\right\Vert _{L_{\infty}\left[0,1\right]}  = \left(E\left(\tau_{1,1}\right)\right)^{-1/2}\cdot\left\Vert \sigma\right\Vert _{L_{\infty}\left[0,1\right]}\quad \mathrm{a.s.}\label{approxi6}
\end{equation}
for all $q\geq 1$ and every equation (\ref{eq1})
\end{Cor}

\section{Appendix} 
As previously, $C$ denotes unspecified positive constants, wich only depend on the constant  $K$ from condition $\mathrm{(A)}$ as
well as on $a\left(0,0\right),\;\sigma\left(0,0\right)$ and $E\bigl|X\left(0\right)\bigr|^{q}$. 

For $\varepsilon >0$ let  $X^{M}_{\varepsilon}$  denotes the process (\ref{milst}) with discretization (\ref{discre2}) \[
0=\tau_{0}<\tau_{1}<\cdots<\tau_{N_{\varepsilon}+1}.\] For $ \ell = 1,\ldots, N_{\varepsilon}$ put 
$$ \Delta_{\ell} = \tau_{\ell}-\tau_{\ell-1} \qquad \mathrm{and} \qquad \Delta_{N_{\varepsilon} + 1} = \xi_{N_{\varepsilon}+1,\varepsilon}.$$ Note, that $\Delta_{\ell}$ is a stopping time with respect to the right-continuous filtration generated by the Brownian motion $\left(W\left(t + \tau_{\ell-1}\right)- W\left( \tau_{\ell-1}\right)\right)_{t\geq 0}$.  We derive  upper bounds for 
$$ \left(E\left\Vert X-\overline{X}_{\varepsilon}\right\Vert _{L_{\infty}\left[0,1\right]}^{q}\right)^{1/q} $$ in terms of $$ C\cdot \left(\varepsilon\cdot \ln\left(1/\varepsilon\right)\right)^{2}.$$ We use this estimate in the analysis of the approximation method (\ref{linearsplineinter1}). In \cite{Slassi:12} an upper bound in the case of deterministic step size has been presented.
\textcompwordmark{}
\begin{Lem}\label{lema1} 
 For all $1\leq q<\infty$ we have 
 \begin{equation} \label{eqlem01}
\left(E\max_{1\leq \ell\leq N_{\varepsilon}+1}\Delta_{\ell}^{q}\right)^{1/q}\leq C\cdot \left(\varepsilon\cdot \ln\left(1/\varepsilon\right)\right)^{2}.
\end{equation}
\end{Lem}

\begin{proof}
At first, from Lemma \ref{LemN} we have for every $\varepsilon >0$
         $$ N_{\varepsilon}+1 \leq \lfloor 1/\varepsilon^{2}\rfloor \qquad \mathrm{a.s.}$$We have
\begin{eqnarray*}
E\left(\max_{1\leq \ell\leq \lfloor 1/\varepsilon^{2}\rfloor}\Delta_{\ell}^{q}\right) &=& \int_{0}^{\infty}\mathbb{P}\left(\max_{1\leq \ell\leq \lfloor 1/\varepsilon^{2}\rfloor }\Delta_{\ell}>t^{1/q}\right)\,dt\\
                                                                                        &=& \int_{0}^{\infty}\left(1-\left(\mathbb{P}\left(\xi_{1,\varepsilon}\leq t^{1/q}\right)\right)^{\lfloor 1/\varepsilon^{2}\rfloor}\right)\,dt\\
                                                                                        &=&\int_{0}^{\infty}\left(1-\left(\mathbb{P}\left(\xi_{1,1}\leq t^{1/q}/\varepsilon^{2}\right)\right)^{\lfloor 1/\varepsilon^{2}\rfloor}\right)\,dt. 
\end{eqnarray*}
Note, that from \cite{CreutzigGronbachRitter:07} we have for all $t\geq 1$ 
\begin{equation}
\mathbb{P}\left(\xi_{1,1}>t\right)\leq\exp\left(-C\cdot t\right) \label{asymptotic10} 
\end{equation}
with some constant $C>0$. \\
Hence
\begin{eqnarray*}  
E\left(\max_{1\leq \ell\leq N_{\varepsilon}+1}\Delta_{\ell}^{q}\right) &\leq& E\left(\max_{1\leq \ell\leq\lfloor 1/\varepsilon^{2}\rfloor} \Delta_{\ell}^{q}\right)\\
                                                                  &\leq& \left(\varepsilon\cdot \ln\left(1/\varepsilon\right)\right)^{2q} + \varepsilon^{2q}\cdot \lfloor 1/\varepsilon^{2}\rfloor \cdot q\cdot \int_{\left(\ln\left(1/\varepsilon\right)\right)^{2}}^{\infty}t^{q-1}\exp\left(-C\cdot t\right)\,dt. \\
                                                                  &\leq&C\cdot\left( \left(\varepsilon\cdot \ln\left(1/\varepsilon\right)\right)^{2q} + \left(\varepsilon\cdot \ln\left(1/\varepsilon\right)\right)^{2q} \cdot\lfloor 1/\varepsilon^{2}\rfloor \cdot \exp\left(-C\cdot\left(\ln\left(1/\varepsilon\right)\right)^{2} \right)\right)\\
                                                                  &\leq & C\cdot \left(\varepsilon\cdot \ln\left(1/\varepsilon\right)\right)^{2q}.
\end{eqnarray*}
This finishes the proof.
\end{proof}
 \begin{Lem} 
\label{lema1-1} For all $1\leq q<\infty$ we have \begin{equation}
E\sup_{t\in\left[0,1\right]}\left|X^{M}_{\varepsilon}\left(t\right)\right|^{q}\leq C\label{lem51}\end{equation}
and \begin{equation}
E\sup_{t\in\left[\tau_{\ell-1},\tau_{\ell}\right]}\left|X^{M}_{\varepsilon}\left(t\right)-X^{M}_{\varepsilon}\left(\tau_{\ell-1}\right)\right|^{2q}\leq C\cdot E\left(\Delta_{\ell}^{q}\right).\label{lem52}\end{equation}
\end{Lem}
\begin{proof}
We may assume that $q\geq2$. Using the linear growth condition (\ref{growth})
it is easy to show by induction on $\ell$ that \begin{equation}
E\left|X^{M}_{\varepsilon}\left(\tau_{\ell}\right)\right|^{q}<\infty\label{eqp1}\end{equation}
for all $\ell\in\left\{1,\ldots,N_{\varepsilon}+1\right\} $.  Put $U_{\ell}=\left(\tau_{\ell},X^{M}_{\varepsilon}\left(\tau_{\ell}\right)\right)$. Then, we have 
\begin{eqnarray*}
\lefteqn{X^{M}_{\varepsilon}\left(t\right) = X\left(0\right)+\int_{0}^{t}\sum_{\ell=1}^{N_{\varepsilon}+1}a\left(U_{\ell-1}\right)\cdot\mathbf{1}_{\left]\tau_{\ell-1},\tau_{\ell}\right]}\left(s\right)\,ds\;}\\
%& &+\int_{0}^{t}\sum_{\ell=1}^{N_{\varepsilon}+1}a\left(U_{\ell-1}\right)\cdot\mathbf{1}_{\left]\tau_{\ell-1},\tau_{\ell}\right]}\left(s\right)\,ds\;\\
 &  & +\int_{0}^{t}\sum_{\ell=1}^{N_{\varepsilon}+1}\sigma\left(U_{\ell-1}\right)\cdot\left(1+\sigma^{\left(0,1\right)}\left(U_{\ell-1}\right)\cdot\left(W\left(s\right)-W\left(\tau_{\ell-1}\right)\right)\right)\cdot\mathbf{1}_{\left]\tau_{\ell-1},\tau_{\ell}\right]}\left(s\right)\, dW\left(s\right).\end{eqnarray*}
Define $Y\left(t\right)=\sup_{s\leq t}\bigl|\overline{X}_{\varepsilon}\left(s\right)\bigr|^{q},\; t\in\left[0,1\right]$. Then, we have 
\begin{eqnarray*}
\lefteqn {E\left(Y\left(t\right)\right) \leq }\\
& & 3^{q-1}\cdot\biggl[E\left|X\left(0\right)\right|^{q}+E\sup_{s\leq t}\left|\int_{0}^{s}\sum_{\ell=1}^{N_{\varepsilon}+1}a\left(U_{\ell-1}\right)\cdot\mathbf{1}_{\left]\tau_{\ell-1},\tau_{\ell}\right]}\left(u\right)\, ds\right|^{q}\;+\\
 &  & E\sup_{s\leq t}\left|\int_{0}^{s}\sum_{\ell=1}^{N_{\varepsilon}+1}\left(\sigma\left(U_{\ell-1}\right)+\left(\sigma\cdot\sigma^{\left(0,1\right)}\right)\left(U_{\ell-1}\right)\cdot\left(W\left(u\right)-W\left(\tau_{\ell-1}\right)\right)\right)\cdot\mathbf{1}_{\left]\tau_{\ell-1},\tau_{\ell}\right]}\left(u\right)\, dW\left(u\right)\right|^{q}\biggr].\end{eqnarray*}
By the Burkholder's and H\"older's inequalities, see \cite[Theorem 3.28]{KaratzasShreve:88},
we get\begin{eqnarray*}
\lefteqn {E\left(Y\left(t\right)\right) \leq } \\ 
&  & c_{q}\cdot\biggl[E\left|X\left(0\right)\right|^{q}+E\left(\int_{0}^{t}\left|\sum_{\ell=1}^{N_{\varepsilon}+1}a\left(U_{\ell-1}\right)\cdot\mathbf{1}_{\left]\tau_{\ell-1},\tau_{\ell}\right]}\left(s\right)\right|\, ds\right)^{q}\;+\nonumber \\
 &  &  E\left(\int_{0}^{t}\left|\sum_{\ell=1}^{N_{\varepsilon}+1}\left(\sigma\left(U_{\ell-1}\right)+\left(\sigma\cdot\sigma^{\left(0,1\right)}\right)\left(U_{\ell-1}\right)\cdot\left(W\left(s\right)-W\left(\tau_{\ell-1}\right)\right)\right)\cdot\mathbf{1}_{\left]\tau_{\ell-1},\tau_{\ell}\right]}\left(s\right)\right|^{2}\, ds\right)^{q/2}\biggr] \\
 & \leq & c_{q}\cdot\biggl[E\left|X\left(0\right)\right|^{q}+t^{q-1}\cdot E\int_{0}^{t}\left|\sum_{\ell=1}^{N_{\varepsilon}+1}a\left(U_{\ell-1}\right)\cdot\mathbf{1}_{\left]\tau_{\ell-1},\tau_{\ell}\right]}\left(s\right)\right|^{q}\, ds\;+\label{eqHoelder}\\
 &  & t^{q/2-1}\cdot E\int_{0}^{t}\left|\sum_{\ell=1}^{N_{\varepsilon}+1}\left(\sigma\left(U_{\ell-1}\right)+\left(\sigma\cdot\sigma^{\left(0,1\right)}\right)\left(U_{\ell-1}\right)\cdot\left(W\left(s\right)-W\left(\tau_{\ell-1}\right)\right)\right)\cdot\mathbf{1}_{\left]\tau_{\ell-1},\tau_{\ell}\right]}\left(s\right)\right|^{q}\, ds\biggr], \end{eqnarray*}
where $c_{q}$ denotes some positive constants depending only on $q$.\\
At first, from the linear growth condition (\ref{growth}) and (\ref{eqp1})
it follows that \begin{equation}
E\left(Y\left(t\right)\right)<\infty\label{eqendlich}\end{equation}
for all $t\in\left[0,1\right]$. We use (\ref{growth}) to obtain$$
E\left(\left|a\left(U_{\ell-1}\right)\right|^{q}\cdot \mathbf{1}_{\left]\tau_{\ell-1},\tau_{\ell}\right]}\left(s\right)\right)\leq C\cdot\left(1+ E\sup_{u\leq s}\left|X^{M}_{\varepsilon}\left(u\right)\right|^{q}\right)$$ and \begin{equation} E\left(\left|\sigma\left(U_{\ell-1}\right)\right|^{q}\cdot \mathbf{1}_{\left]\tau_{\ell-1},\tau_{\ell}\right]}\left(s\right)\right)\leq C\cdot\left(1+ E\sup_{u\leq s}\left|X^{M}_{\varepsilon}\left(u\right)\right|^{q}\right).\label{abs2-1}
\end{equation}
Thus  
\begin{eqnarray}
\lefteqn{ E\left( \left|\sigma\left(U_{\ell-1}\right)+\left(\sigma\cdot\sigma^{\left(0,1\right)}\right)\left(U_{\ell-1}\right)\cdot\left(W\left(s\right)-W\left(\tau_{\ell-1}\right)\right)\right|^{q}\cdot\mathbf{1}_{\left]\tau_{\ell-1},\tau_{\ell}\right]}\left(s\right)\right) \leq }\nonumber\\
                         & & 2^{q-1}\Bigl[E\left|\sigma\left(U_{\ell-1}\right)\right|^{q}+ C\cdot E\left|\sigma\left(U_{\ell-1}\right)\right|^{q}\cdot E\sup_{\tau_{\ell-1}\leq s\leq \tau_{\ell}}\left|W\left(s\right)-W\left(\tau_{\ell-1}\right)\right|^{q}\Bigr] \nonumber\\
                         & \leq & C\cdot\left(1+E\sup_{u\leq s}\left|X^{M}_{\varepsilon}\left(u\right)\right|^{q}\right).\label{abs1-1}\end{eqnarray} 
%\begin{eqnarray}
%E\left( \left|\sigma\left(U_{\ell-1}\right)+\left(\sigma\cdot\sigma^{\left(0,1\right)}\right)\left(U_{\ell-1}\right)\cdot\left(W\left(s\right)-W\left(\tau_{\ell-1}\right)\right)\right|^{q}\cdot\mathbf{1}_{\left]\tau_{\ell},\t%au_{\ell+1}\right]}\left(s\right)\right) & \leq & 2^{q-1}\bigl[E\left|\sigma\left(U_{\ell-1}\right)\right|^{q}+ \nonumber\\
%                         &  & C\cdot E\left|\sigma\left(U_{\ell-1}\right)\right|^{q}\cdot E\sup_{\tau_{\ell-1}\leq t\leq \tau_{\ell}}\left|W\left(s\right)-W\left(\tau_{\ell-1}\right)\right|^{q}\bigr]\nonumber \\
% & \leq & C\cdot\left(1+E\sup_{u\leq s}\left|\overline{X}_{\varepsilon}\left(u\right)\right|^{q}\right).\label{abs1-1}\end{eqnarray} 
Hence, from (\ref{abs2-1}) and (\ref{abs1-1}) we get \[
E\left(Y\left(t\right)\right)\leq C\cdot\left[E\left|X\left(0\right)\right|^{q}+t^{q/2-1}\left(t+\int_{0}^{t}E\left(Y\left(s\right)\right)\, ds\right)\right],\]
and result (\ref{lem51}) follows from (\ref{eqendlich}) by Gronwall's
lemma.\\ Furthermore, we have
\begin{eqnarray*}
\lefteqn{\sup_{t\in\left[\tau_{\ell-1},\tau_{\ell}\right]}\left|X^{M}_{\varepsilon}\left(t\right)-X^{M}_{\varepsilon}\left(\tau_{\ell-1}\right)\right|^{2q}\;  \leq\;  C\cdot\Biggl[\bigl(1+\sup_{t\in\left[0,1\right]}\left|X^{M}_{\varepsilon}\left(t\right)\right|^{2q}\bigr)\cdot \biggl(|\tau_{\ell}-\tau_{\ell-1}|^{2q} \;+}\\
 &  & \sup_{\tau_{\ell-1}\leq t\leq \tau_{\ell}}\left|W\left(t\right)-W\left(\tau_{\ell-1}\right)\right|^{2q}+ \sup_{\tau_{\ell-1}\leq t\leq \tau_{\ell}}\left|W\left(t\right)-W\left(\tau_{\ell-1}\right)\right|^{4q}\biggr) \Biggr].
\end{eqnarray*} 
From this we obtain the second estimate (\ref{lem52}) by using (\ref{lem51})
and Burkholder's inequality.\end{proof} 
 
Let $$
  F\left(s\right) = \sum_{\ell=1}^{N_{\varepsilon}+1}\left(a^{\left(0,1\right)}\cdot\sigma\right)\left(U_{\ell-1}\right)\cdot \left(W\left(s\right)-W\left(\tau_{\ell-1}\right)\right)\cdot\mathbf{1}_{\left]\tau_{\ell-1},\tau_{\ell}\right]}\left(s\right). $$

\begin{Lem}\label{lema1-2} For all $1\leq q<\infty$  we have 
\begin{equation}
E\sup_{s\leq 1}\left|\int_{0}^{s}F\left(u\right)\, du\right|^{q}\leq C\cdot\left(\left(E \max_{1\leq \ell\leq N_{\varepsilon}+1}\Delta_{\ell}^{2q}\right)^{1/2} + \varepsilon^{2q}\right). \label{eqlema1-2}
\end{equation}
\end{Lem}  
\begin{proof}
We may assume that $q=2p$ with $p\in \mathbb{N}$. At first, we define \[
n_{t}=\max\left\{ \ell=1,\ldots,N_{\varepsilon}+1\;:\; \tau_{\ell-1}\leq t\right\} ,\]
and we consider the right continuous filtration $\left(\mathcal{F}_{t}\right)_{t\in\left[0,1\right]}$ generated by $W$. 
%given by \begin{equation}
%\mathcal{F}_{t}=\sigma\left(W\left(s\right);\;0\leq s\leq n_{t}\right).\label{filtration}\end{equation}
%For $s\in\left]t_{\ell},t_{\ell+1}\right]$ we have $F\left(s\right)=V_{\ell}\left(s\right)$.
Then, we have\[
E\sup_{s\leq 1}\left|\int_{0}^{s}F\left(u\right)\, du\right|^{q}\leq2^{q-1}\left[E\sup_{s\leq 1}\left|A_{1}\left(s\right)\right|^{q}+E\sup_{s\leq 1}\left|A_{2}\left(s\right)\right|^{q}\right],\]
where \begin{eqnarray*}
A_{1}\left(s\right) & = & \int_{0}^{\tau_{n_{s}}}\sum_{\ell=1}^{N_{\varepsilon}+1}\left(a^{\left(0,1\right)}\cdot\sigma\right)\left(U_{\ell-1}\right)\cdot\left(W\left(u\right)-W\left(\tau_{\ell-1}\right)\right)\cdot\mathbf{1}_{\left]\tau_{\ell-1},\tau_{\ell}\right]}\left(u\right)\, du\\
A_{2}\left(s\right) & = & \int_{s}^{\tau_{n_{s}}}\left(a^{\left(0,1\right)}\right)\cdot\sigma\left(U_{n_{s}-1}\right)\cdot\left(W\left(u\right)-W\left(\tau_{n_{s}-1}\right)\right)\, du.\\
\end{eqnarray*}
 We use the linear growth condition (\ref{growth}) and (\ref{lem51}) as well as H\"older's inequality 
to obtain\begin{eqnarray}
E\sup_{s\leq 1}\left|A_{2}\left(s\right)\right|^{q} & \leq & C\cdot E\sup_{s\leq 1}\left|\sigma\left(U_{n_{s}-1}\right)\cdot\int_{s}^{\tau_{n_{s}}}\left(W\left(u\right)-W\left(\tau_{n_{s}-1}\right)\right)\, du\right|^{q}\nonumber\\
 & \leq & C\cdot E\sup_{s\leq 1}\left(\left(\tau_{n_{s}}-s\right)^{q}\cdot \left|\sigma\left(U_{n_{s}-1}\right)\right|^{q}\cdot \sup_{s \leq u\leq \tau_{n_{s}}}\left|W\left(u\right)-W\left(\tau_{n_{s}-1}\right)\right|^{q}\right)\nonumber\\
 & \leq & C\cdot E\max_{1\leq\ell\leq N_{\varepsilon}+1}\left(\Delta_{\ell}^{q}\cdot \left(1+\left|X^{M}_{\varepsilon}\left(\tau_{\ell-1}\right)\right|^{q}\right) \cdot \sup_{\tau_{\ell-1}\leq u\leq \tau_{\ell}}\left|W\left(u\right)-W\left(\tau_{\ell-1}\right)\right|^{q}\right)\nonumber\\
 & \leq & C\cdot  \left(E\max_{1\leq\ell\leq N_{\varepsilon}+1}\Delta_{\ell}^{2q}\right)^{1/2}.\label{eqA1}\end{eqnarray}
We put\[
\Theta_{\ell}=\int_{\tau_{\ell-1}}^{\tau_{\ell}}\left(W\left(u\right)-W\left(\tau_{\ell-1}\right)\right)\, du.\]Not, that from Remark \ref{remark} we have for all $p\in \mathbb{N}$ \[
E\left(\Theta_{\ell}\right)^{p}=k_{p}\cdot E\left(\Delta_{\ell}^{3p/2}\right),\]
where$$ 
k_{p}\begin{cases}
= 0 & : \mathrm{if} \;p\;\mathrm{odd}\\
> 0 & :\mathrm{else}.\end{cases}$$
 Using this fact it is not difficult to show that $\bigl(A_{1}\left(s\right)\bigr)_{s\in\left[0,1\right]}$
is a right-continuous martingale with respect to the filtration $\left(\mathcal{F}_{\tau_{n_{s}}}\right)_{s\in\left[0,1\right]}$. So, by Doob's maximal inequality we have
\begin{eqnarray*}
E\sup_{s\leq 1}\left|A_{1}\left(s\right)\right|^{2p} &\leq & c_{p}\cdot E\left(A_{1}\left(1\right)\right)^{2p}\\
                                                     &=& c_{p}\cdot E\left(\sum_{\ell =1}^{N_{\varepsilon}+1} \left(a^{\left(0,1\right)}\cdot\sigma\right)\left(U_{\ell-1}\right)\cdot \Theta_{\ell}\right)^{2p}. 
\end{eqnarray*}
For $n\in \mathbb{N}$ we put $$ S_{n} =\sum_{\ell =1}^{n} \left(a^{\left(0,1\right)}\cdot\sigma\right)\left(U_{\ell-1}\right)\cdot \Theta_{\ell} \qquad \mathrm{and}\qquad S_{0}=0.$$
For every $\varepsilon >0$ and  $n\in\{0, \dots, \lfloor 1/\varepsilon^{2}\rfloor \}$ let $ T_{n}= \left(N_{\varepsilon}+1\right) \land n$. Then, we have 
$$ S_{T_{n+1}} =S_{T_{n}} +\left(S_{n+1}-S_{n}\right)\cdot \mathbf{1}_{\{N_{\varepsilon}+1>n\}} = S_{T_{n}} +\left(a^{\left(0,1\right)}\cdot\sigma\right)\left(U_{n}\right)\cdot \Theta_{n+1}\cdot \mathbf{1}_{\{N_{\varepsilon}+1>n\}}.$$
Hence,
\begin{eqnarray*}
\lefteqn{ E\left(|S_{T_{n+1}}|^{2p}\big | \mathcal{F}_{\tau_{n}}\right)}\\
  & = & E\left( \left(S_{T_{n}} + \left(a^{\left(0,1\right)}\cdot\sigma\right)\left(U_{n}\right)\cdot \Theta_{n+1}\cdot \mathbf{1}_{\{N_{\varepsilon}+1>n\}}\right)^{2p}\Big | \mathcal{F}_{\tau_{n}}\right) \\
 &= & \sum_{r=0}^{2p}{2p \choose r}\cdot E\left( \left(S_{T_{n}}\right)^{2p-r}\cdot\left(\left(a^{\left(0,1\right)}\cdot\sigma\right)\left(U_{n}\right)\cdot \Theta_{n+1}\cdot \mathbf{1}_{\{N_{\varepsilon}+1>n\}}\right)^{r}\Big | \mathcal{F}_{\tau_{n}}\right)\\
 &=&  \sum_{r=0}^{2p}{2p \choose r} \left(S_{T_{n}}\right)^{2p-r}\cdot\left(\left(a^{\left(0,1\right)}\cdot\sigma\right)\left(U_{n}\right)\right)^{r} \cdot \mathbf{1}_{\{N_{\varepsilon}+1>n\}}\cdot E\left(\Theta_{n+1}\right)^{r} \\
 &\leq & \left(S_{T_{n}}\right)^{2p} + C\cdot \sum_{r=1}^{p}{2p \choose 2r}\left(S_{T_{n}}\right)^{2p-2r}\cdot\left(\left(a^{\left(0,1\right)}\cdot\sigma\right)\left(U_{n}\right)\right)^{2r}\cdot \varepsilon^{6r}\\
  &\leq & \left(S_{T_{n}}\right)^{2p} + C\cdot \varepsilon^{2}\cdot \sum_{r=1}^{p}{2p \choose 2r}\left(S_{T_{n}}\right)^{2p-2r}\cdot\left(\varepsilon^{2}\cdot\left(1+\left|X^{M}_{\varepsilon}\left(\tau_{n}\right)\right|\right)\right)^{2r}.
\end{eqnarray*}
We use Lemma \ref{lema1-1} to obtain
\begin{eqnarray*}
\lefteqn{ E\left( \left(S_{T_{n}}\right)^{2p-2r}\cdot\left(1+\left|X^{M}_{\varepsilon}\left(\tau_{n}\right)\right|\right)^{2r}\right)}\\
  &\leq & \left(E\left(S_{T_{n}}\right)^{2p}\right)^{\frac{2p-2r}{2p}}\cdot \left(E\left(1+\left|X^{M}_{\varepsilon}\left(\tau_{n}\right)\right|\right)^{2p}\right)^{\frac{2r}{2p}}\\
  & \leq & C\cdot\left(E\left(S_{T_{n}}\right)^{2p}\right)^{\frac{2p-2r}{2p}}.  
\end{eqnarray*}
Thus
\begin{eqnarray*} 
\lefteqn{ E|S_{T_{n+1}}|^{2p}}\\
 &\leq &  E|S_{T_{n}}|^{2p} + C\cdot \varepsilon^{2}\cdot \sum_{r=1}^{p}{2p \choose 2r} \left(E|S_{T_{n}}|^{2p}\right)^{\frac{2p-2r}{2p}}\cdot \varepsilon^{4r}\\
 &\leq & E|S_{T_{n}}|^{2p} + C\cdot \varepsilon^{2}\cdot\left(\left(E|S_{T_{n}}|^{2p}\right)^{\frac{1}{2p}} + \varepsilon^{2}\right)^{2p}\\
 &\leq & E|S_{T_{n}}|^{2p}\cdot\left(1+C\cdot\varepsilon^{2}\right) + C\cdot\varepsilon^{4p+2}.
\end{eqnarray*} 
 Now observe that $T_{\lfloor 1/\varepsilon^{2}\rfloor} = N_{\varepsilon} +1$  and apply a discrete version of Gronwall's Lemma (see \cite{GronbachT:02}) to obtain
\begin{equation}\label{eqA2}
E\sup_{s\leq 1}\left|A_{1}\left(s\right)\right|^{2p} \leq C\cdot \varepsilon^{4p}.
\end{equation}
Finally, use (\ref{eqA1}) and (\ref{eqA2}) as well as Lemma \ref{lema1} to complete the proof.
\end{proof}
\begin{Thm}\label{theorem} Let $X$ be the solution of (\ref{eq1}). Then, for all $1\leq q<\infty$ we have
$$ \left(E\left\Vert X-X^{M}_{\varepsilon}\right\Vert _{L_{\infty}\left[0,1\right]}^{q}\right)^{1/q}\leq C\cdot \left(\varepsilon\cdot \ln\left(1/\varepsilon\right)\right)^{2}.$$
\end{Thm}
\begin{proof}
We assume without loss of generality that $q=2p$ with $p\in\mathbb{N}$. 
Put $U_{\ell}=\bigl(\tau_{\ell},X^{M}_{\varepsilon}\bigl(\tau_{\ell}\bigr)\bigr)$.
We compute the difference $X\left(t\right)-X^{M}_{\varepsilon}\left(t\right)$
\begin{align*}
X\left(t\right)-X^{M}_{\varepsilon}\left(t\right) = \int_{0}^{t}\sum_{\ell=1}^{N_{\varepsilon}+1}\left(a\left(s,X\left(s\right)\right)-a\left(U_{\ell-1}\right)\right)\cdot\mathbf{1}_{\left]\tau_{\ell-1},\tau_{\ell}\right]}\left(s\right)\, ds\; + \qquad \qquad \quad \quad \qquad \qquad \quad  &  \qquad \qquad \quad \quad  \\
%\int_{0}^{t}\sum_{\ell=1}^{N_{\varepsilon}+1}\left(a\left(s,X\left(s\right)\right)-a\left(U_{\ell-1}\right)\right)\cdot\mathbf{1}_{\left]\tau_{\ell-1},\tau_{\ell}\right]}\left(s\right)\, ds\;&\\
 \int_{0}^{t}\sum_{\ell=1}^{N_{\varepsilon}+1}\left(\sigma\left(s,X\left(s\right)\right)-\sigma\left(U_{\ell-1}\right)-\left(\sigma\cdot\sigma^{\left(0,1\right)}\right)\left(U_{\ell-1}\right)\cdot\left(W\left(s\right)-W\left(\tau_{\ell-1}\right)\right)\right)\cdot\mathbf{1}_{\left]\tau_{\ell-1},\tau_{\ell}\right]}\left(s\right)\, dW\left(s\right)&
\end{align*} 
%\begin{eqnarray*}
%\lefteqn{X\left(t\right)-\overline{X}_{\varepsilon}\left(t\right)=}\\ 
%&  & \int_{0}^{t}\sum_{\ell=1}^{N_{\varepsilon}+1}\left(a\left(s,X\left(s\right)\right)-a\left(U_{\ell-1}\right)\right)\cdot\mathbf{1}_{\left]\tau_{\ell-1},\tau_{\ell}\right]}\left(s\right)\, ds\;+\\
% &  & \int_{0}^{t}\sum_{\ell=1}^{N_{\varepsilon}+1}\left(\sigma\left(s,X\left(s\right)\right)-\sigma\left(U_{\ell-1}\right)-\left(\sigma\cdot\sigma^{\left(0,1\right)}\right)\left(U_{\ell-1}\right)\cdot\left(W\left(s\right)-W\%left(\tau_{\ell-1}\right)\right)\right)\cdot\mathbf{1}_{\left]\tau_{\ell-1},\tau_{\ell}\right]}\left(s\right)\, dW\left(s\right).
%\end{eqnarray*}
Define \begin{eqnarray*}
Z\left(t\right) & = & \sup_{s\leq t}\left|X\left(s\right)-X^{M}_{\varepsilon}\left(s\right)\right|\\
V_{\ell}\left(s\right) & = & \left(a^{\left(0,1\right)}\cdot\sigma\right)\left(U_{\ell}\right)\cdot \left(W\left(s\right)-W\left(\tau_{\ell}\right)\right)\cdot\mathbf{1}_{\left]\tau_{\ell-1},\tau_{\ell}\right]}\left(s\right)\\
F\left(s\right) & = & \sum_{\ell=1}^{N_{\varepsilon}+1}V_{\ell-1}\left(s\right)\cdot\mathbf{1}_{\left]\tau_{\ell-1},\tau_{\ell}\right]}\left(s\right)\\
\check{V}_{\ell}\left(s\right) & = & \sigma^{\left(0,1\right)}\left(U_{\ell}\right)\cdot\left(X^{M}_{\varepsilon}\left(s\right)-X^{M}_{\varepsilon}\left(t_{\ell}\right)-a\left(U_{\ell}\right)\cdot\left(s-\tau_{\ell}\right)\right)\cdot\mathbf{1}_{\left]\tau_{\ell-1},\tau_{\ell}\right]}\left(s\right)\\
R_{\ell}\left(s\right) & = & \left(\sigma^{\left(0,1\right)}\cdot\sigma\right)\left(U_{\ell}\right)\cdot\int_{\tau_{\ell}}^{s}\sigma^{\left(0,1\right)}\left(U_{\ell}\right)\cdot\left(W\left(u\right)-W\left(\tau_{\ell}\right)\right)\, dW\left(u\right)\cdot\mathbf{1}_{\left]\tau_{\ell-1},\tau_{\ell}\right]}\left(s\right)\\
\check{F}\left(s\right) & = & \sum_{\ell=1}^{N_{\varepsilon}+1}R_{\ell-1}\left(s\right)\cdot\mathbf{1}_{\left]\tau_{\ell-1},\tau_{\ell}\right]}\left(s\right).\end{eqnarray*}
Similar to (\ref{eqHoelder}), we use Burkholder's and H\"older's inequalities
to obtain \begin{eqnarray*}
E\left(Z\left(t\right)\right)^{q} & \leq & C\cdot\Bigl[Et^{q-1}\int_{0}^{t}\left|\sum_{\ell=1}^{N_{\varepsilon}+1}\left(a\left(s,X\left(s\right)\right)-a\left(U_{\ell-1}\right)-V_{\ell-1}\left(s\right)\right)\cdot\mathbf{1}_{\left]\tau_{\ell-1},\tau_{\ell}\right]}\left(s\right)\right|^{q}\, ds\\
 &  & +E\sup_{s\leq t}\left|\int_{0}^{s}F\left(u\right)\, du\right|^{q}\\
 &  & +Et^{q/2-1}\cdot\int_{0}^{t}\left|\sum_{\ell=1}^{N_{\varepsilon}+1}\left(\sigma\left(s,X\left(s\right)\right)-\sigma\left(U_{\ell-1}\right)-\check{V}_{\ell-1}\left(s\right)\right)\cdot\mathbf{1}_{\left]\tau_{\ell-1},\tau_{\ell}\right]}\left(s\right)\right|^{q}\, ds\\
 &  & +Et^{q/2-1}\cdot\int_{0}^{t}\left|\check{F}\left(s\right)\right|^{q}\, ds\Bigr].\end{eqnarray*}
Using a similar argument as in proof of Proposition 1 in the Appendix of \cite{GronbachT:02},  
we get by Lemma \ref{lema1-1} that \begin{equation}
E \left( \left|\sigma\left(s,X\left(s\right)\right)-\sigma\left(U_{\ell-1}\right)-\check{V}_{\ell-1}\left(s\right)\right|^{q} \cdot\mathbf{1}_{\left]\tau_{\ell-1},\tau_{\ell}\right]}\left(s\right)\right)\leq C\cdot\left(E\Delta_{\ell}^{q}+E\left(Z\left(s\right)\right)^{q}\right)\label{proofeq1}\end{equation}
and \begin{equation}
E\left( \left|a\left(s,X\left(s\right)\right)-a\left(U_{\ell-1}\right)-V_{\ell-1}\left(s\right)\right|^{q}\cdot\mathbf{1}_{\left]\tau_{\ell-1},\tau_{\ell}\right]}\left(s\right)\right) \leq C\cdot \left(E\Delta_{\ell}^{q}+E\left(Z\left(s\right)\right)^{q}\right)\label{proofeq3}\end{equation} 
Furthermore, we have
\begin{eqnarray*}
\lefteqn {E \left(\left|R_{\ell-1}\left(s\right)\right|^{q}\cdot\mathbf{1}_{\left]\tau_{\ell-1},\tau_{\ell}\right]}\left(s\right)\right) }\\
& \leq & C\cdot E\left|\sigma\left(U_{\ell}\right)\right|^{q}\cdot E\sup_{\tau_{\ell-1}\leq s\leq \tau_{\ell}}\left| \left(\left(W\left(s\right)-W\left(\tau_{\ell-1}\right)\right)^{2}-\left(s-\tau_{\ell-1}\right)\right)\right|^{q}\\
 & \leq & C\cdot  E\max_{1\leq \ell\leq N_{\varepsilon}+1}\Delta_{\ell}^{q}
\end{eqnarray*}
 by the linear growth condition (\ref{growth}), Lemma \ref{lema1-1} and  Burkholder's inequality. Thus 
\begin{equation}
E\int_{0}^{t}\left|\check{F}\left(s\right)\right|^{q}\, ds\leq C\cdot t\cdot E\max_{1\leq \ell\leq N_{\varepsilon}+1}\Delta_{\ell}^{q}.\label{proofeq4}
\end{equation} 
Combining (\ref{proofeq1})-(\ref{proofeq4}) and Lemma \ref{lema1} as well as Lemma \ref{lema1-2}, we get \[
E\left(Z\left(t\right)\right)^{q}\leq C\cdot\left(t^{q/2-1}\int_{0}^{t}E\left(Z\left(s\right)\right)^{q}\, ds+\varepsilon^{2q}+ \left(\varepsilon\cdot \ln\left(1/\varepsilon\right)\right)^{2q}\right).\]
By (\ref{eqEX}) and (\ref{lem51}) we have \[  
E\left(Z\left(t\right)\right)^{q}<\infty\]
for all $t\in\left[0,1\right]$, and therefore, by Gronwall's Lemma\[
E\left(Z\left(t\right)\right)^{q}\leq C\cdot \left(\varepsilon\cdot\ln\left(1/\varepsilon\right)\right)^{2q}.\]
This concludes the proof of the Theorem.
\end{proof}
\begin{Cor}\label{kor1} For all $\kappa \in \left(0, 1\right)$ there exist a nonnegative random variable $\zeta_{\kappa}$ with $E\left(\left|\zeta_{\kappa}\right|^{q}\right) <\infty$ for all $1\leq q<\infty$ so that
$$ \left\Vert X-X^{M}_{\varepsilon}\right\Vert _{L_{\infty}\left[0,1\right]}\leq \zeta_{\kappa}\cdot \varepsilon^{2-\kappa}\qquad \mathrm{a.s.}$$
\end{Cor}
\begin{proof}
The assertions is a direct consequence from Theorem \ref{theorem} and Lemma 2.1 in \cite{KloedenNeuenkirch:07}.  
\end{proof}

\bibliographystyle{plain}
\bibliography{references}
%\bibliography{spline} 

%\printindex{}

\end{document}